\documentclass[psamsfonts,fceqn,leqno]{amsart}
\usepackage{mathrsfs,latexsym,amsfonts,amssymb,curves,epic}
\usepackage{color}

\newtheorem{theorem}{Theorem}[section]
\newtheorem{corollary}[theorem]{Corollary}
\newtheorem{proposition}[theorem]{Proposition}
\newtheorem{lemma}[theorem]{Lemma}
\newtheorem{question}[theorem]{Question}
\theoremstyle{definition}
\newtheorem{definition}[theorem]{Definition}
\newtheorem{remark}[theorem]{Remark}
\newtheorem{example}[theorem]{Example}
\newtheorem{problem}[theorem]{Problem}

\def\N{{\mathbb N}}
\def\Q{{\mathbb Q}}

\begin{document}
\noindent \vspace{0.5in}

\title[Countable tightness and $\mathfrak G$-bases on Free topological groups]
{Countable tightness and $\mathfrak G$-bases on Free topological groups}

  \author{Fucai Lin}
  \address{(Fucai Lin): School of mathematics and statistics, Minnan Normal University, Zhangzhou 363000, P. R. China}
  \email{linfucai2008@aliyun.com; linfucai@mnnu.edu.cn}
    \author{Alex Ravsky}
  \address{(Alex Ravsky): Pidstrygach Institute for Applied Problems of Mechanics and Mathematics of NASU, Naukova 3b,
Lviv, 79060, Ukraine}
  \email{oravsky@mail.ru}
    \author{Jing Zhang}
  \address{(Jing Zhang): School of mathematics and statistics, Minnan Normal University, Zhangzhou 363000, P. R. China}
  \email{zhangjing86@126.com}
\thanks{The first author is supported by the NSFC (Nos. 11571158, 11201414, 11471153),
  the Natural Science Foundation of Fujian Province (Nos. 2016J05014, 2016J01671, 2016J01672) of China
 and the project of Abroad Fund of Minnan
  Normal University. This paper was partially
  written when the first author was visiting the School of Computer and
  Mathematical Sciences at Auckland University of Technology from March
  to September 2015, and he wishes to thank the hospitality of his host.}

%insert keywords
%insert subject class

%\date{\today}
\begin{abstract}
 Given a Tychonoff space $X$, let $F(X)$ and $A(X)$ be respectively the free
  topological group and the free Abelian topological group over $X$ in the sense
  of Markov. In this paper, we consider two topological properties of $F(X)$ or $A(X)$, namely the countable tightness
  and $\mathfrak G$-base. We provide some characterizations of the
  countable tightness and $\mathfrak G$-base of $F(X)$ and $A(X)$ for various
  special classes of spaces $X$. Furthermore, we also study the countable tightness
  and $\mathfrak G$-base of some $F_{n}(X)$ of $F(X)$.
\end{abstract}
\maketitle
\section{Introduction}
Let $F(X)$ and $A(X)$ be respectively the free topological group
and the free Abelian topological group over a Tychonoff space $X$
in the sense of Markov \cite{MA1945}. For every $n\in\mathbb{N}$, by $F_{n}(X)$ we denote the subspace of $F(X)$ that consists of all words of reduced length at most $n$ with respect to the free basis $X$. The subspace
$A_{n}(X)$ is defined similarly.
We always use $G(X)$ to denote $F(X)$ or $A(X)$, and $G_{n}(X)$ to $F_{n}(X)$ or $A_{n}(X)$ for each $n\in \mathbb{N}$. Therefore, any statement about $G(X)$ applies to $F(X)$ and $A(X)$, and about $G_{n}(X)$ applies to $F_{n}(X)$ and $A_{n}(X)$.

One of the techniques of studying the topological structure of free topological groups is
to clarify the relations of subspaces $X$, $F(X)$, $A(X)$, $F_{n}(X)$ and $A_{n}(X)$, where $n\in\mathbb{N}$.
It is well known that only when the space $X$ is discrete, $F(X)$ and $A(X)$ can be first-countable. Therefore, the space $F(X)$ is first-countable if and only if $X$ is discrete \cite{G1962}. Similarly, the groups $F(X)$ and $A(X)$ are locally compact if and only if the space $X$ is discrete \cite{D961}. More generally, P. Nickolas and M. Tkachenko proved that if one of the groups
$F(X)$ or $A(X)$ is almost metrizable, then the space $X$ is discrete \cite{T2003}.
Further, K. Yamada gave a characterization for a metrizable space $X$ such that some the spaces $F_{n}(X)$ and $A_{n}(X)$ are first-countable \cite{Y1998}.

Recently, Z. Li et al. in \cite{LLL} proved that for each stratifiable $k$-space, the group $F(X)$ is of countable tightness if and only if the space $X$ is separable or discrete. In Section 3, we refine this result by giving a characterization of a space
$X$ such that the countable tightness of the space $F_{8}(X)$ implies the countable tightness of the group $F(X)$. Furthermore, since each space with the countable fan-tightness or the strong Pytkeev property is of countable tightness, we also discuss the topological properties of the countable fan-tightness and the strong Pytkeev property of free topological group $F(X)$ or some $F_{n}(X)$.

Ferrando et al. in \cite{F2006} introduced the concept of $\mathfrak G$-base in the frame of locally convex spaces. Now the concept of $\mathfrak G$-base plays an important role in the study of function spaces, see \cite{B2, CVHT2014, GKL, GKL2, GKL3, GKL1, GKL4, LPT2015}. From \cite{GKL2}, we know that the strong Pytkeev property for general topological groups is closely related to the notion of a $\mathfrak G$-base. For instance, each topological group which is a $k$-space with a $\mathfrak G$-base has the strong Pytkeev property. In Section 4, we shall continuously discuss the properties of free topological groups with a $\mathfrak G$-base, which are motivated by the following interesting Questions~\ref{q0} and \ref{q1}.

 \begin{question}\cite[Question 4.17]{GKL}\label{q0}
Let $X$ be a submetrizable $k_{\omega}$-space. Does the group $F(X)$ have a $\mathfrak G$-base?
\end{question}

\begin{question}\cite[Question 4.15]{GKL}\label{q1}
For which $X$ the groups $A(X)$ and $F(X)$ have a $\mathfrak G$-base?
\end{question}

Recently, S.S. Gabriyelyan and J. K\c{a}kol in the paper \cite{GKL3} and
A.G. Leiderman, V.G. Pestov , A.H. Tomita in the paper \cite{LPT2015} have given an answer to Questions~\ref{q0} and~\ref{q1} respectively.

\maketitle
\section{Notations and Terminology}
In this section, we introduce the necessary notations and terminology. Throughout this paper, all topological spaces are assumed to be
  Tychonoff, unless otherwise is explicitly stated. For undefined notations and terminology, refer to \cite{AT2008}, \cite{E1989} and \cite{Gr}.
  First of all, let $\N$ and $\Q$ denote the sets of positive integers and rational numbers, respectively.

 \medskip
Let $X$ be a topological space and $A \subseteq X$.
The \emph{closure} of a subspace $A$ of $X$ is denoted by $\overline{A}$. The subspace $A$ is called
\emph{bounded} if every continuous real-valued function $f$ defined on the subspace $A$ is bounded. If the closure of every bounded
set in $X$ is compact, then the space $X$ is called \emph{$\mu$}-complete. The space $X$ is called a \emph{$cf$-space} if every compact subset of
  it is finite. The space $X$ is called a
\emph{$k$-space} provided that a subset $C\subseteq X$ is closed in $X$ if
and only if $C\cap K$ is closed in $K$ for each compact subset $K$ of the space $X$. In particular, the space $X$ is called a \emph{$k_{\omega}$-space} if there exists a family of countably many compact subsets $\{K_{n}: n\in\mathbb{N}\}$ of $X$ such that each subset $F$ of the space $X$ is closed in $X$ if and only if $F\cap K_{n}$ is closed in $K_{n}$ for each $n\in\mathbb{N}$. A subset $A$ of the space $X$ is \emph{sequentially open} if each sequence $\{x_{n}\}_{n\in_{\N}}$ in $X$ converging to a point of $A$ is eventually in $A$. The space $X$ is called \emph{sequential} if every sequentially open subset of $X$ is open. A space $X$ is of \emph{countable tightness} if whenever $A\subset X$ and $x\in \overline{A}$, there exists a
countable set $B\subset A$ such that $x\in \overline{B}$. A space $X$ is of \emph{countable fan-tightness} \cite{A1986} if for any countable family $\{A_{n}: n\in\N\}$ of
subsets of $X$ satisfying $x\in\bigcap_{n\in\N}\overline{A_{n}}$, it is possible to select a finite set $K_{n}\subset A_{n}$ for each $n\in\mathbb{N}$, in such
a way that $x\in\overline{\bigcup_{n\in\N} K_{n}}$. A sequence $\{x_n\}$, convergent to a point $x$ of $X$, is called non-trivial, provided all points $x_n$ and $x$ are
mutually distinct.

Let $\mathscr{P}$ be a family of subsets of a space $X$. The family $\mathscr{P}$ is called a {\it network} at a point $x\in X$ if for each open neighborhood $U$ of $x$ in $X$ there exists an element $P\in \mathscr{P}$ such that $x\in P\subset U$. The family $\mathscr P$ is called a
\emph{$cs$-network} \cite{G1971} at a point $x\in X$ if whenever a sequence
$\{x_n: n \in \N\}$ converges to the point $x$ and $U$ is an arbitrary open neighborhood
of the point $x$ in $X$ there exist a number $m\in\N$ and an element $P\in
\mathscr P$ such that
\[
 \{x\}\cup \{x_n: n\geqslant m\} \subseteq P \subseteq U.
 \]
The space $X$ is called \emph{csf-countable} if $X$ has a countable $cs$-network
at each point $x\in X$. We call the family $\mathscr P$ a \emph{$cs^{\ast}$-network} at a point $x\in X$
\cite{LT1994} of $X$ if whenever a sequence $\{x_n: n\in\N\}$ converges to the
point $x$ and $U$ is an arbitrary open neighborhood of the point $x$ in $X$, there
are an element $P\in\mathscr P$ and a subsequence $\{x_{n_{i}}: i\in \N\}$ of
$\{x_n: n\in \N\}$ such that
$\{x\}\cup\{x_{n_{i}}: i\in \mathbb{N}\} \subseteq P\subseteq U$.
Furthermore, the family $\mathscr P$ is called a {\it $k$-network}
\cite{O1971} if whenever $K$ is a compact subset of $X$ and $U\subset X$ is an arbitrary open set containing $K$ then there is a finite subfamily $\mathscr {P}^{\prime}\subseteq
\mathscr {P}$ such that $K\subseteq \bigcup\mathscr {P}^{\prime}\subseteq U$.
Recall that a regular space $X$ is
{\it $\aleph_{0}$} if
$X$ has a countable
$k$-network.
The family $\mathscr{P}$ is called a {\it Pytkeev network} \cite{P1983} at a point $x\in X$ if $\mathscr{P}$ is a network at $x$ and for every open set $U$ in $X$ and a set $A$ accumulating at $x$ there exists $P\in\mathscr{P}$ such that $P\subset U$ and $P\cap A$ is infinite; the family $\mathscr{P}$ is a {\it Pytkeev network} in $X$ if $\mathscr{P}$ is a Pytkeev network at each point $x\in X$. The space $X$ is said to have the \emph{strong Pytkeev property} \cite {TZ2009} if at each point of $X$ there is a countable Pytkeev network. The space $X$ is called {\it $\mathfrak{P}_{0}$} if $X$ is regular and has a countable Pytkeev network.

The following
implications follow directly from definitions. However, none of them can be reversed. By \cite[Proposition 4.1]{GKL4}, we see that a space is first-countable if and only if it has the strong Pytkeev property and is of countable fan-tightness.

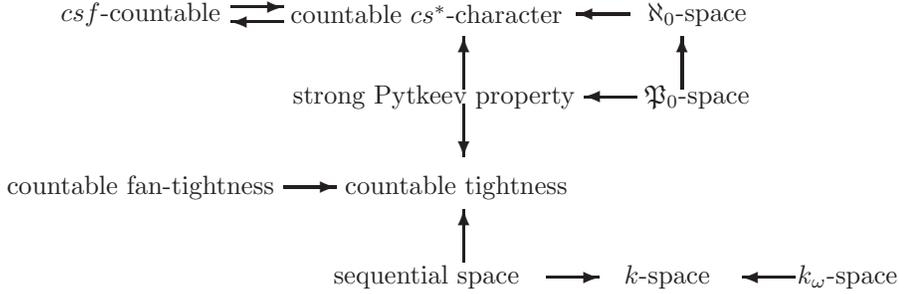
\setlength{\unitlength}{1cm}
\begin{picture}(15,3.2)\thicklines
 \put(7.4, -0.7){\vector(1,0){0.7}}
 \put(3.9, 0.5){\vector(1,0){0.7}}
 \put(10.7, -0.7){\vector(-1,0){0.7}}
 \put(5.8,-0.7){\makebox(0,0){sequential space}}
 \put(9,-0.7){\makebox(0,0){$k$-space}}
 \put(6.2,0.5){\makebox(0,0){countable tightness}}
 \put(11.4,-0.7){\makebox(0,0){$k_{\omega}$-space}}
 \put(2,0.5){\makebox(0,0){countable fan-tightness}}
 \put(6.3, -0.5){\vector(0,1){0.7}}
 \put(6.3, 1.6){\vector(0,-1){0.7}}
 \put(5.9,1.7){\makebox(0,0){strong Pytkeev property}}
 \put(9.4,1.7){\makebox(0,0){$\mathfrak{P}_{0}$-space}}
 \put(8.6, 1.7){\vector(-1,0){0.7}}
 \put(9.2, 1.8){\vector(0,1){0.7}}
 \put(9.4,2.8){\makebox(0,0){$\aleph_{0}$-space}}
 \put(8.5, 2.8){\vector(-1,0){0.7}}
 \put(5.8,2.8){\makebox(0,0){countable $cs^{\ast}$-character}}
 \put(6.3, 1.8){\vector(0,1){0.7}}
 \put(2,2.8){\makebox(0,0){$csf$-countable}}
 \put(3.2, 2.9){\vector(1,0){0.7}}
 \put(3.9, 2.7){\vector(-1,0){0.7}}
\end{picture}

\vskip 1cm\setlength{\parindent}{0.5cm}

\begin{definition}\cite{Gr}
A topological space $X$ is a {\it stratifiable space} if $X$ is $T_{1}$ and, to each open $U$ in $X$, on can assign a sequence $\{U_{n}\}_{n=1}^{\infty}$ of open subsets of $X$ such that

(a) $\overline{U_{n}}\subset U$;

(b) $\bigcup_{n=1}^{\infty}U_{n}=U;$

(c) $U_{n}\subset V_{n}$ whenever $U\subset V$.
\end{definition}

{\bf Note:} Clearly, each metrizable space is stratifiable \cite{Gr}.

We consider the product $\N^{\N}$ with the natural partial order, i.e., $\alpha\leq\beta$ if $\alpha_{i}\leq\beta_{i}$ for each $i\in\N$, where $\alpha=(\alpha_{i})_{i\in\N}$ and $\beta=(\beta_{i})_{i\in\N}$. A topological space $(X, \tau)$ has a {\it small base} \cite{GKL1} if there exist a subset $M$ of $\N^{\N}$
 and a family of open subsets $\mathscr{U}=\{U_{\alpha}: \alpha\in M\}$ in $X$ such that $\mathscr{U}$ is a base for $X$ and $U_{\beta}\subset U_{\alpha}$ for all $\alpha, \beta\in M$ with $\alpha\leq\beta$. In particular, we say that $(X, \tau)$ has a $\mathfrak G$-base \cite{F2006} if $M=\N^{\N}$.
If a space has a $\mathfrak G$-base, then it has a countable $cs^{\ast}$-character, see Proposition~\ref{p00}.

 \medskip
  Given a group $G$, the letter $e_G$ denotes the neutral element of $G$. If no
  confusion occurs, we simply use $e$ instead of $e_G$ to denote the neutral
  element of $G$.

  \medskip
  Let $X$ be a non-empty Tychonoff space. Throughout this paper, $X^{-1}
  :=\{x^{-1}: x\in X\}$ and $-X: =\{-x: x\in X\}$, which are copies of
  $X$. Let $e$ be the neutral element of $F(X)$ (i.e., the empty
  word) and $0$ be that of $A(X)$. For every $n\in\N$ and an element
  $(x_{1}, x_{2}, \cdots, x_{n})$ of $(X\bigoplus X^{-1}\bigoplus\{e\})^{n}$
  we call a word $g=x_{1}x_{2}\cdots x_{n}$ a {\it form}. In the Abelian case,
 a word $x_{1}+x_{2}\cdots +x_{n}$ is also called a {\it form} for $(x_{1}, x_{2},
  \cdots, x_{n})\in(X\bigoplus -X\bigoplus\{0\})^{n}$. This word $g$ is
  called {\it reduced} if it does not contains $e$ or any pair of
  consecutive symbols of the form $xx^{-1}$ or $x^{-1}x$. It follows
  that if the word $g$ is reduced and non-empty, then it is different
  from the neutral element $e$ of $F(X)$. In particular, each element
  $g\in F(X)$ distinct from the neutral element can be uniquely written
  in the form $g=x_{1}^{r_{1}}x_{2}^{r_{2}}\cdots x_{n}^{r_{n}}$, where
  $n\geq 1$, $r_{i}\in\mathbb{Z}\setminus\{0\}$, $x_{i}\in X$, and
  $x_{i}\neq x_{i+1}$ for each $i=1, \cdots, n-1$, and the {\it support}
  of $g=x_{1}^{r_{1}} x_{2}^{r_{2}}\cdots x_{n}^{r_{n}}$ is defined as
  $\mbox{supp}(g) :=\{x_{1}, \cdots, x_{n}\}$. Given a subset $K$ of
  $F(X)$, we put $\mbox{supp}(K)=\bigcup_{g\in K}\mbox{supp}(g)$. Similar
  assertions (with the obvious changes for commutativity) are valid for
  $A(X)$. For every $n\in\mathbb{N}$, let $i_n: (X\bigoplus X^{-1}
  \bigoplus\{e\})^{n} \to F_n(X)$ be the natural mapping defined by
  $i_n(x_1, x_2, ... x_n)= x_1x_2...x_n$
  for each $(x_1, x_2, ... x_n) \in (X\bigoplus X^{-1} \bigoplus\{e\})^{n}$.
  We also use the same symbol in the Abelian case, that is, it means
  the natural mapping from $(X\bigoplus (-X)\bigoplus\{0\})^{n}$ onto
  $A_{n}(X)$. Clearly, each $i_n$ is a continuous mapping.

\maketitle
\section{The characterization of countable tightness in free topological groups}
In this section, we mainly discuss the countable tightness and countable fan-tightness of free topological groups. First, we give a characterization of a stratifiable $k$-space $X$ such that $F_8(X)$ has countable tightness (or,
equivalently, $F(X)$ has countable tightness).  Then we show that a space $X$ must belong to some special class of spaces if $F_{4}(X)$ is of countable fan-tightness.

The following theorem generalizes a result in \cite{LLL}.

\begin{theorem}\label{ttightness}
Let $X$ be a stratifiable $k$-space. Then the following are equivalent:
\begin{enumerate}
\item $F_{8}(X)$ is of countable tightness;

\item $F(X)$ is of countable tightness;

\item The space $X$ is separable or discrete.
\end{enumerate}
\end{theorem}

\begin{proof}
Since the equivalence (2) $\Leftrightarrow$ (3) was proved in \cite{LLL}, it suffices to show that (1) $\Rightarrow$ (3).

Assume that $X$ is neither separable nor discrete. Since each stratifiable space has $G_\delta$-diagonal, by \v Sne\v\i der Theorem \cite{Sn}, each compact subspace of $X$ is metrizable. Thus $X$ is sequential. Since the space $X$ is non-discrete, it contains a non-isolated point $x\in X$.
This means that the set $\{x\}$ is not sequentially open, that is, there exists a non-trivial convergent sequence. Hence take an arbitrary a non-trivial convergent sequence $C:=\{x_{n}: n\in\mathbb{N}\}\subset X$ with a limit point $x$. Moreover, assume that the space $X$ contains no uncountable closed discrete subset. This means that the extend
of the space $X$ is countable. But, by \cite{K1985},
each stratifiable space is a $\sigma$-space, and a $\sigma$-space of countable extent is cosmic
(that is, has countable network), and, therefore, separable. Obtained contradiction shows that there exists
an uncountable discrete closed subset $D:=\{d_{\alpha}: \alpha\in\omega_{1}\}$ of $X$. Without loss of generality, we may assume that $C\cap D=\emptyset$.

 For each $\alpha\in\omega_{1}$ let $f_{\alpha}: \omega_{1}\rightarrow\omega$ be a function such that $f_{\alpha}\upharpoonright_{\alpha}: \alpha\rightarrow\omega$ is a bijection. For any distinct $\alpha, \beta\in\omega_{1}$, put $$E_{\alpha, \beta}:=\{d_{\beta}x_{m}x^{-1}d_{\beta}^{-1}d_{\alpha}x_{f_{\alpha}(\beta)}x^{-1}d_{\alpha}^{-1}: m\leq f_{\alpha}(\beta)\}.$$ Then put $$E:=\bigcup_{\alpha, \beta\in\omega_{1}, \alpha\neq\beta} E_{\alpha, \beta}.$$ Clearly, $e\not\in E$, and $e\in\overline{E}$ by the proof of \cite[Proposition 2.2]{Y2005}. In order to obtain a contradiction, it suffices to show that each countable infinite subset $B\subset E$ is closed in $F(X)$.  Let $Y:=\mbox{supp}(B)$. The set $Y$ contains the point $x$, so $Y\subset C\cup D\cup\{x\}$, which implies that $F(Y)$ is a $k$-space by \cite[Theorem 3.7]{AOP1989}.  Since $X$ is a stratifiable space and $Y$ is closed in $X$, it follows from \cite{S2000} that the subgroup $F(Y, X)$ of $F(X)$ generated  by $Y$ is naturally topologically isomorphic to $F(Y)$. Then $F(Y)$ is a closed $k$-subspace in $F(C\bigoplus D)$. Furthermore, we claim that for each compact subset $K$ of $F_{8}(Y)$, the set $K\cap B$ is finite. Assume on the contrary that there is a compact subset $K$ in $F_{8}(Y)$ such that $K\cap B$ is infinite. Clearly, the set $K\cap B$ is a bounded subset in $F_{8}(Y)$, hence the subspace $\mbox{supp}(K\cap B)$ is bounded in $Y$ by \cite[Theorem 1.5]{AOP1989}. Since the space $Y$ is paracompact, $\overline{\mbox{supp}(K\cap B)}$ is compact in $Y$. However, the set $\mbox{supp}(K\cap B)$ contains infinite many elements $d_{\alpha}'s$ since $B$ is infinite, which is a contradiction with the compactness of the subspace $\overline{\mbox{supp}(K\cap B)}$. Therefore, the subset $B$ is closed in $F_{8}(Y)$, that is, the subset $B$ is closed in $F(X)$. Hence $F_{8}(X)$ is not of countable tightness since $e\in\overline{E}$, which is a contradiction.
\end{proof}

Obviously, we have the following corollary.

\begin{corollary}
Let $X$ be a stratifiable $k$-space. If $F_{8}(X)$ has the strong Pytkeev property, then $X$ is separable or discrete.
\end{corollary}

 \begin{remark}\label{r0}
Let $X:=D\bigoplus K$, where $D$ is an uncountable discrete space
  and $K$ is an infinite compact metric space. Then
  $F_{3}(X)$ is first-countable by Theorem 4.5 in \cite{Y1998}, hence it is of countable tightness. However, the space $X$ is not separable and discrete. We do not know whether $F_{4}(X)$ is of countable tightness. Therefore, we have the following Question~\ref{qcountable}.
  \end{remark}

  \begin{question}\label{qcountable}
Let $X:=C\bigoplus D$, where $C$ is a non-trivial convergent sequence with its limit point
  and $D$ is an uncountable discrete space.  For each $n\in\{4, 5, 6, 7\}$, does $F_{n}(X)$ have the countable tightness?
  \end{question}

Furthermore, we also do not know the answer to the following question.

  \begin{question}\label{qf2}
Let $X$ be a space.  If $F_{2}(X)$ is of countable tightness, does $F_{3}(X)$ have the countable tightness?
  \end{question}

It is natural to ask whether Theorem~\ref{ttightness} holds in the class of free Abelian topological groups. Next we shall give a partial answer to this question.

\begin{theorem}\label{abeltightness}
Let $X$ be a stratifiable $k$-space. If $A_{4}(X)$ is of countable tightness, then the set of all non-isolated points of $X$ is a separable subspace in $X$.
\end{theorem}

\begin{proof}
Assume on the contrary that the set of all non-isolated points of $X$ is not separable. Then the set of all non-isolated points of $X$ is not an $\aleph_{1}$-compact space since $X$ is a stratifiable space \cite{Gr}. Therefore, there exists an uncountable closed discrete subset $D:=\{d_{\alpha}: \alpha<\omega_{1}\}$ in $X$, where each $d_{\alpha}$ is a non-isolated point in $X$. Moreover, since $X$ is a stratifiable $k$-space, it is paracompact (and, hence, collectionwise normal by ~\cite[Th.5.1.17]{E1989}) and sequential. Hence there exists a family of mutually disjoint open subsets $\{U_{\alpha}: \alpha<\omega_{1}\}$ of $X$ such that $d_{\alpha}\in U_{\alpha}$ for each $\alpha<\omega_{1}$. Moreover, for each $\alpha<\omega_{1}$ we can take a non-trivial sequence
$\{d_{\alpha}(n): n\in\N\}\subset U_{\alpha}$ convergent to a point $d_{\alpha}$. Let $C(\omega_{1}):=\bigoplus_{\alpha<\omega_{1}}(\{d_{\alpha}(n): n\in\N\}\cup\{d_{\alpha}\})$. Since $X$ is stratifiable and $C(\omega_{1})$ is closed in $X$, it follows from \cite{S2000} that the subgroup $A(C(\omega_{1}), X)$ of $A(X)$ generated  by $C(\omega_{1})$ is naturally topologically isomorphic to the free Abelian topological group $A(C(\omega_{1}))$. Then $A_{4}(C(\omega_{1}))$ is topologically isomorphic to $A(C(\omega_{1}), X)\cap A_{4}(X)$, which implies that the tightness of $A_{4}(C(\omega_{1}))$ is countable. However, it follows from \cite[Theorem 3.2]{Y1997} and \cite{G1982} that the tightness of $A_{4}(C(\omega_{1}))$ is uncountable, which is a contradiction.
\end{proof}

However, the converse of Theorem~\ref{abeltightness} does not hold, see \cite[Theorem 3.2]{Y1997}.

Moreover, the proof of the following result is similar to \cite[Proposition 2.2]{Y2005}, and thus, the proof is omitted in this paper.

\begin{theorem}
For a stratifiable space $X$, if $F_{8}(X)$ is a $k$-space, then $X$ is separable or discrete.
\end{theorem}

Next, we shall discuss the countable fan-tightness in free topological groups. In contrast to Theorem~\ref{ttightness}, we shall find that the situation changes dramatically for the countable fan-tightness in free topological groups. First, we shall give a characterization of a space $X$ such that $G(X)$ has the countable fan-tightness.

\begin{theorem}
Let $X$  be a space. Then $G(X)$ has the countable fan-tightness if and only if the space $X$ is discrete.
\end{theorem}

\begin{proof}
We only consider $F(X)$, as the proof of the case of $A(X)$ is quite similar. Since the sufficiency is obvious, we shall prove only the necessity. In order to obtain a contradiction, assume the converse. Suppose that $F(X)$ has a countable fan-tightness
but the space $X$ is not discrete. Clearly, $e\in\bigcap_{n\in\N}\overline{(\bigcup_{i\geq n}(F_{i}(X))\setminus F_{i-1}(X)))}$. However, if for each natural $n$ we take an arbitrary finite subset $D_{n}\subset \bigcup_{i\geq n}(F_{i}(X))\setminus F_{i-1}(X)))$,
then an intersection $(\bigcup_{n\in\N}D_{n})\cap F_{m}(X)$ will be finite for each natural $m$. Therefore $\bigcup_{n\in\N}D_{n}$ is closed and discrete in $F(X)$ by \cite[Corollary 7.4.3]{AT2008}. Thus $e\not\in\overline{\bigcup_{n\in\N}D_{n}}$, a contradiction.
\end{proof}

It turns out that the countable fan-tightness of $F_{4}(X)$ imposes strong restrictions on the space $X$. Recall that a subspace $Y$ of a space $X$ is said to be {\it P-embedded in $X$} if each continuous pseudometric
on $Y$ admits a continuous extensions over $X$.

\begin{theorem}\label{tcountable fan}
Let $X$ be a space. If $F_{4}(X)$ has the countable fan-tightness, then $X$ is either pseudocompact or a $cf$-space.
\end{theorem}

\begin{proof}
Suppose $X$ is not a $cf$-space. Then there exists an infinite compact subset $C$ in $X$.
Next we shall show that $X$ is pseudocompact.

Assume the converse. Then there exists a discrete family $\{U_{n}: n\in\N\}$ of non-empty open subsets of the space $X$. It can be easily verified that the family $\{\overline{U_{n}}: n\in\N\}$ is also discrete, hence $\bigcup_{n\in\N}\overline{U_{n}}$ is closed in $X$. Since the set $C$ is compact, it can intersect at most finitely many $\overline{U_{n}}$'s. Thus without loss of generality, we may assume that $C\cap \bigcup_{n\in\N}\overline{U_{n}}=\emptyset$. Then the family $\{C\}\cup\{U_{n}: n\in\N\}$ is discrete in $X$.

Since $C$ is an infinite compact set, it contains a non-isolated point $x$. For each $n\in\N$, pick $y_{n}\in U_{n}$, and put $C_{n}:=y_{n}^{-1}x^{-1}C y_{n}$. Let $$Y:=C\cup\{y_{n}: n\in\N\}, \mbox{and}\ Z:=\bigcup_{n\in\N} C_{n}.$$

Obviously, the set $Y$ is closed, $\sigma$-compact and non-discrete in $X$. Moreover, $Y$ is $P$-embedded in $X$ by \cite{T1984}. By \cite{S2000}, the subgroup $F(Y, X)$ of $F(X)$ generated  by $Y$ is naturally topologically isomorphic to $F(Y)$. Moreover, since $Y$ is a $k_{\omega}$-space, it follows from \cite[Theorem 7.4.1]{AT2008} that $F(Y)$ is also a $k_{\omega}$-space. Hence $F_{4}(Y)$ is also a $k_{\omega}$-space. Next, we claim that $Z$ is closed in $F_{4}(Y)$.

Indeed, for each $n\in \mathbb{N}$, let $$K_{n}:=C\cup C^{-1}\cup\{y_{i}, y_{i}^{-1}: i\leq n\}\cup\{e\}$$and $$X_{n}:=\overbrace{K_{n}\cdot\ldots \cdot K_{n}}^{n}.$$Then it follows from the proof of \cite[Theorem 7.4.1]{AT2008} that the topology of $F(Y)$ is determined by the family of compact subsets $\{X_{n}: n\in\mathbb{N}\}$, hence the topology of $F_{4}(Y)$ is determined by the family $\{X_{n}\cap F_{4}(Y): n\in\mathbb{N}\}$. Therefore, we have $Z\cap X_{n}\cap F_{4}(Y)=\bigcup_{i=1}^{n}C_{i}$ for each $n\in\mathbb{N}$, which shows that $Z$ is closed in $F_{4}(Y)$. Therefore, $Z$ has the countable fan-tightness. Moreover, it is easy to see that $e\in\overline{C_{n}\setminus\{e\}}$ for each $n\in\N$, that is, $e\in\bigcap_{n\in\N}\overline{C_{n}\setminus\{e\}}$. Now, if for each natural $n$, $F_{n}\subset C_{n}\setminus\{e\}$ is a finite subset,
then $e\not\in\overline{\bigcup_{n\in\N}F_{n}}=\bigcup_{n\in\N}F_{n}$, because the topology of $Z$
is determined by the family $\{C_{n}: n\in\mathbb{N}\}$, a contradiction.

\end{proof}

\begin{corollary}
Let $X$ be a $\mu$- and $k$-space. If $F_{4}(X)$ has the countable fan-tightness, then $X$ is either compact or discrete.
\end{corollary}

\begin{remark}
By Theorems~\ref{ttightness} and~\ref{tcountable fan}, it is easy to see that $F(C\bigoplus D)$ has the countable tightness and $F_{4}(C\bigoplus D)$ does not have the countable fan-tightness, where $C$ is a non-trivial convergent sequence with its limit point and $D$ is a countable infinite discrete space. However, $F_{3}(C\bigoplus D)$ has the countable fan-tightness by Remark~\ref{r0}. Therefore, we can not replace ``$F_{4}(X)$'' with ``$F_{3}(X)$'' in Theorem~\ref{tcountable fan}. Furthermore, $F(C\bigoplus D)$ is a $\mathfrak{P}_{0}$-space since $F(C\bigoplus D)$ is an $k$-space with a countable $k$-network.
\end{remark}

Finally, we shall discuss the strong Pytkeev property in free topological groups. It is well known that if a space has the strong Pytkeev property then it is of countable tightness and is $csf$-countable. In \cite{B1}, the authors showed that a space is first-countable if and only if it has the strong Pytkeev property and countable fan-tightness. Therefore, it is interesting to discuss the strong Pytkeev property on free topological groups. First, we shall give a theorem which has just been proved in \cite{LL}.

 Recall that a space $X$ is said to be
  \emph{La\v{s}nev} if it is the closed image of some metric space.

\begin{theorem}\cite{LL}\label{t00}
 Let $X$ be a non-discrete La\v{s}nev space. Then $F_{4}(X)$ is of $csf$-countable
if and only if $F(X)$ is an $\aleph_0$-space.
\end{theorem}

 \begin{theorem}\label{thm:_csf}
  Let $X$ be a non-discrete La\v{s}nev space. Then $F(X)$ has the strong Pytkeev property
if and only if $F(X)$ is a $\mathfrak{P}_{0}$-space.
  \end{theorem}

  \begin{proof}
Obviously, it suffices to show the necessity. Suppose that $F(X)$ has the strong Pytkeev property. By Theorem~\ref{t00}, the group $F(X)$ is an $\aleph_{0}$-space, hence $F(X)$ is separable. It follows from \cite[Theorem 1.7]{GKL3} that $F(X)$ is a $\mathfrak{P}_{0}$-space.
  \end{proof}

We do not know whether an $\aleph_{0}$-space with the the strong Pytkeev property is a $\mathfrak{P}_{0}$-space. If this answer is positive, then we can replace ``$F(X)$ has the strong Pytkeev property'' by ``$F_{4}(X)$ has the strong Pytkeev property'' in Theorem~\ref{thm:_csf}.

In \cite{B}, T. Banakh posed the following problem.

\begin{problem}\cite{B}
Let $X$ be a (sequential) $\mathfrak{P}_{0}$-space. Is the free topological group over $X$ a $\mathfrak{P}_{0}$-space?
\end{problem}

But we even do not know an answer to the following question.

\begin{question}
Let $X$ be the rational number subspace $\mathbb{Q}$ with the usual topology. Is the free topological group over $X$ a $\mathfrak{P}_{0}$-space?
\end{question}

 However, we have the following result.

\begin{theorem}
Let $X$ be a $\mathfrak{P}_{0}$-space. Then $F(X)$ is the union of countably many $\mathfrak{P}_{0}$-subspaces.
\end{theorem}

\begin{proof}
For each $n\in \mathbb{N}$, it follows from \cite[Corollary 3.12]{B} that $(X\bigoplus X^{-1}\bigoplus\{e\})^{n}$ is a $\mathfrak{P}_{0}$-space. For each $n\in \mathbb{N}$, since the mapping $$i_{n}\upharpoonright_{i_{k}^{-1}(F_{n}(X)\setminus F_{n-1}(X))}: i_{n}^{-1}(F_{n}(X)\setminus F_{n-1}(X))\rightarrow F_{n}(X)\setminus F_{n-1}(X)$$ is a homeomorphism, it follows from \cite[Corollary 3.12]{B} that $i_{n}^{-1}(F_{n}(X)\setminus F_{n-1}(X))$ is a $\mathfrak{P}_{0}$-subspace in $(X\bigoplus X^{-1}\bigoplus\{e\})^{n}$, hence $F_{n}(X)\setminus F_{n-1}(X)$ is a $\mathfrak{P}_{0}$-subspace in $F(X)$. Since $F(X)=\bigcup_{n\in \mathbb{N}}(F_{n}(X)\setminus F_{n-1}(X))$, $F(X)$ is the union of disjoint countably many $\mathfrak{P}_{0}$-subspaces.
\end{proof}

For closing this section, we discuss the strong Pytkeev property in topological spaces.
It is well known that in the class of regular countably compact spaces the property of countable tightness is equivalent to the countable fan-tightness \cite[Corollary 2]{AB1996}. Moreover, each compact sequential non-first-countable space is of countable tightness, hence it does not have the strong Pytkeev property. Hence in the class of regular countably compact spaces the property of countable tightness is not equivalent to the strong Pytkeev property.
Moreover, it is well known that there exists a countably compact non-metrizable space with a point-countable $k$-network. It is natural to ask whether in the class of regular countably compact spaces the existence of a point-countable
$k$-network implies the existence of a point-countable Pytkeev network.  The answer is also negative, see Example~\ref{e0}.

\begin{example}\label{e0}
There exists an infinite countably compact space $X$ which satisfies the following conditions:
\begin{enumerate}
\item The space $X$ contains no infinite compact subset;

\item The space $X$ has a point-countable $k$-network;

\item The space $X$ does not have the strong Pytkeev property.
\end{enumerate}
\end{example}

\begin{proof}
By \cite{F1960}, there exists an infinite, countably
compact subspace $X$ of $\beta\mathbb{N}$, the Stone-Cech compactification of the
of the space of natural numbers endowed with discrete topology, such that $\mathbb{N}\subset X$ and every compact subset of the space $X$ is finite. Therefore, the space $X$ has a point-countable $k$-network. However, the space $X$ does not have the strong Pytkeev property, see \cite{B}.
\end{proof}

From the opposite side, recently, Z. Cai and S. Lin has proved that each sequentially compact
space with a point-countable $k$-network is metrizable \cite{C2015}. Remark, that
the proof of \cite[Proposition 1.4]{B1}, implies that the space $X$ has a point-countable $k$-network
provided it has a point-countable Pytkeev network. So the next theorem is a counterpart of this result.

\begin{theorem}
Let $X$ be a Hausdorff countably compact space with a point-countable Pytkeev network. Then the space $X$ is a metrizable compact space.
\end{theorem}

\begin{proof}
Let $\mathcal N$ be a point-countable Pytkeev network on the space $X$. We first show the following claim.

{\bf Claim:} For each countably compact subset $K$ of $X$ and arbitrary open subset $U$ with $K\subset U$, there exists a finite subfamily $\mathcal{F}$ such that $K\subset\bigcup\mathcal{F}\subset U$.

Suppose not. For each $x\in K$, let $\{N\in\mathcal{N}: x\in N\}=\{N_{n}(x): n\in\omega\}$. Inductively choose $x_{n}\in K$ such that $x_{n}\not\in N_{j}(x_{i})$ for $i, j<n$. Put $A:=\{x_{n}: n\in\omega\}$. Since $K$ is countably compact and $A\subset K$, the set $A$ has a cluster point $x^{\star}$ in $K$. By the definition of Pytkeev network, it follows that there exists some $N\in\mathcal{N}$ such that $N$ contains infinitely many $x_{n}$'s. Therefore, we have $N=N_{j}(x_{i})$ for some $i$ and $j$, and there exists $n>i, j$ such that $x_{n}\in N_{j}(x_{i})$, contradicting the way the $x_{n}$'s were chosen.

By Claim and \cite[Theorem 4.1]{GMT1}, the space $X$ is metrizable (and thus compact). 
\end{proof}

\section{Free topological groups with a $\mathfrak G$-base}
In this section, we shall discuss the properties of free topological groups with a $\mathfrak G$-base, which are motivated by Questions~\ref{q0} and \ref{q1}. First, we recall a lemma. Then we shall give a characterization of free topological groups which are $k$-spaces having a $\mathfrak G$-base.
 Let $\textbf{TG}_{\mathfrak G}$ be the class of all topological groups having a $\mathfrak G$-base.

\begin{lemma}\label{l0}\cite{GKL}
Let $G\in\textbf{TG}_{\mathfrak G}$. Then the following are equivalent:
\begin{enumerate}
\item The group $G$ is a $k$-space;

\item The group $G$ is a sequential space;

\item The group $G$ is metrizable or contains a submetrizable open $k_{\omega}$-subgroup.
\end{enumerate}
\end{lemma}

By Lemma~\ref{l0}, we know that the $k$-property and sequentiality are equivalent in the class of all topological groups having a $\mathfrak G$-base. In \cite{GKL}, the authors also said that ``It would be interesting to know whether the $k$-property and sequentiality are equivalent for the class of all topological groups having countable $cs^{\ast}$-character''. Indeed, {\bf the answer is negative}, see the following example.

\begin{example}\label{e1}
There exists a topological group $G$ such that it is a $k$-space with a countable $cs^{\ast}$-character. However, $G$ is not sequential.
\end{example}

\begin{proof}
Let $X$ be the Stone-\v{C}ech compactification $\beta D$ of any infinite discrete space $D$. Let $G:=F(X)$ or $G:=A(X)$. Obviously, the group $G$ is a $k$-space by \cite[Theorem 7.4.1]{AT2008}. It follows from a result of \cite{LL} that $G$ is of countable $cs^{\ast}$-character. However, it is well known that $\beta D$ is not a sequential space. Since $\beta D$ is closed in $G$, the free topological group $G
$ is not sequential.
\end{proof}

However, the topological group $G$ in the proof of Example~\ref{e1} does not have the strong Pytkeev property by Example~\ref{e0}, hence it is natural to pose the following question.

\begin{question}
Let a topological group $G$ be a $k$-space. If $G$ has the strong Pytkeev property, is it sequential?
\end{question}

In \cite{GKL3}, the authors gave an affirmative answer to Question~\ref{q0}. The following theorem complements it.

\begin{theorem}\label{t submetrziable}
Let $X$ be a space. Then $F(X)$ is a $k$-space with a $\mathfrak G$-base
if and only if either $X$ is discrete or $X$ is a submetrizable $k_{\omega}$-space.
\end{theorem}

\begin{proof}
The sufficiency was proved in \cite{GKL3}. It suffices to show the necessity.

Let $F(X)$ be a $k$-space with a $\mathfrak G$-base. Then it follows from Lemma~\ref{l0} that $F(X)$ is metrizable or contains a submetrizable open $k_{\omega}$-subgroup. If $F(X)$ is metrizable, then it is well known that $X$ is discrete. Hence we may assume that $F(X)$ is non-metrizable, and then $F(X)$ contains a submetrizable open $k_{\omega}$-subgroup. Then $F(X)=\bigoplus_{\alpha\in\Gamma} G_{\alpha}$, where each $G_{\alpha}$ is a submetrizable open $k_{\omega}$-subset in $F(X)$. Since each $G_{\alpha}$ has a countable $k$-network, it follows that $F(X)$ has a $\sigma$-compact finite $k$-network. Moreover, it is obvious that $F(X)$ is locally Lindel\"{o}f. It is well known that a locally Lindel\"{o}f topological group is paracompact \cite[Problem 3.2.A]{AT2008}, then $X$ is paracompact since $X$ is closed in $F(X)$. Thus $X$ is a paracompact space with a $\sigma$-compact finite $k$-network. Since $F(X)$ is a non-metrizable $k$-space, it follows from \cite[Theorem 4.14]{LLL} that $X$ has a countable $k$-network, hence $X$ is Lindel\"{o}f and submetrizable. Therefore, it is easy to see that $X$ is a submetrizable $k_{\omega}$-space.
\end{proof}

\begin{remark}\label{r000}
However, there exists a space $X$ such that $A(X)$ is a $k$-space with a $\mathfrak G$-base and $X$ is not a Lindel\"{o}f-space. Indeed, let $X:=C\bigoplus D$, where $C$ is a non-trivial convergent sequence with its limit point and $D$ is an uncountable discrete space $D$. Then $A(X)\cong A(C)\times A(D)$, thus $A(X)$ is a $k$-space. Since both $A(X)$ and $A(D)$ have $\mathfrak G$-bases, it follows from \cite{GKL} that $A(S)\times A(D)$ has a $\mathfrak G$-base. However, it is obvious that $X$ is not Lindel\"{o}f, hence it is not a $k_{\omega}$-space.
\end{remark}

By Remark~\ref{r000}, we see that we can not replace ``$F(X)$'' by ``$A(X)$'' in Theorem 4.4. However, we have the following theorem when we add some additional assumption on the space $X$.

\begin{theorem}
Let $X$ be a separable space. Then $A(X)$ is a $k$-space with a $\mathfrak G$-base if and only if $X$ is either countable discrete or a submetrizable $k_{\omega}$-space.
\end{theorem}

\begin{proof}
We adapt the proof of Theorem~\ref{t submetrziable} for the group $A(X)$ instead of $F(X)$. It suffices to show that $X$ is a submetrizable $k_{\omega}$-space if $A(X)$ is a non-metrizable $k$-space with a $\mathfrak G$-base. Since $X$ is separable, $A(X)$ is separable. Similarly to the proof of Theorem~\ref{t submetrziable}, we see that the index set $\Gamma$ in Theorem~\ref{t submetrziable}
is countable. Therefore, $A(X)$ has a countable $k$-network, and then $X$ is a submetrizable $k_{\omega}$-space.
\end{proof}

Next we consider the topological properties of $X$ such that the free topological group over $X$ has a $\mathfrak{G}$-base.

By the proof of \cite[Theorem 3.12]{GKL}, we can easily obtain the following proposition.

\begin{proposition}\label{p00}
If a space $X$ has a $\mathfrak{G}$-base at point $x\in X$, then $X$ is of countable $cs^{\ast}$-character at $x$.
\end{proposition}

Therefore, we have the following proposition.

\begin{proposition}\label{l00}
Let $X$ be a space. If each $G_{n}(X)$ has a $\mathfrak{G}$-base at point $e$, then $G(X)$ is $csf$-countable.
\end{proposition}

\begin{proof}
It suffices to note that for each compact subset $K$ in $G(X)$ there exists an $n\in\mathbb{N}$ such that $K\subset G_{n}(X),$ see
\cite[Corollary 7.4.4]{AT2008}.
\end{proof}

 An answer to the following question is still unknown for us.

\begin{question}
Let $X$ be a space. If each $G_{n}(X)$ has a $\mathfrak{G}$-base at $e$, does $G(X)$ have a $\mathfrak{G}$-base?
\end{question}

\begin{theorem}\label{t0}
Let $X$ be a collectionwise normal space containing a non-trivial convergent sequence. If $F(X)$ has a $\mathfrak{G}$-base, then $X$ is $\aleph_{1}$-compact.
\end{theorem}

\begin{proof}
Suppose that $X$ is not $\aleph_{1}$-compact. Hence there exists an uncountable closed discrete subset $D$ in $X$. Moreover, by the assumption, there exists a non-trivial convergent sequence $S$ with its limit point in $X$. Without loss of generality, we may assume that $S\cap D=\emptyset.$ Let $Y=S\cup D$. Since $X$ is collectionwise normal, the subspace $Y$ is a retract of $X$, and then $Y$ is $P$-embedded in $X$ \cite[Exercises 7.7.a]{AT2008}. By \cite{S2000}, the subgroup $F(Y, X)$ of $F(X)$ generated  by $Y$ is naturally topologically isomorphic to $F(Y)$. However, $F(Y)$ is not of $csf$-countable by a result in \cite{LL}, thus $F(Y)$ is not $cs^{\ast}$-countable. Then $F(X)$ is not $cs^{\ast}$-countable. However, since $F(X)$ has a $\mathfrak{G}$-base, it follows from Proposition~\ref{p00} that $F(X)$ is $cs^{\ast}$-countable, which is a contradiction.
\end{proof}

\begin{corollary}
Let $X$ be a stratifiable $k$-space. If $F(X)$ has a $\mathfrak{G}$-base, then $X$ is either discrete or separable.
\end{corollary}

\begin{proof}
Assume that $X$ is not discrete. Since a stratifiable $k$-space is paracompact and sequential, $X$ is $\aleph_{1}$-compact by Theorem~\ref{t0}. By \cite{K1985}, each stratifiable space is a $\sigma$-space, and each $\aleph_{1}$-compact $\sigma$-space is cosmic, and, therefore, separable.
\end{proof}

Recently, A.G. Leiderman, V.G. Pestov and A.H. Tomita in \cite{LPT2015} showed the following two results:

\begin{theorem}\label{t1}\cite{LPT2015}
The free Abelian topological group $A(X)$ on a uniform
space $X$ has a $\mathfrak{G}$-base if and only if $X$ has a $\mathfrak{G}$-base.
\end{theorem}

\begin{corollary}\label{c0}\cite{LPT2015}
Let $X$ be a metrizable space and the set of all non-isolated points of $X$ is a $\sigma$-compact subset of $X$. Then $A(X)$ has a $\mathfrak{G}$-base.
\end{corollary}

For a metrizable space $X$, it follows from \cite{AOP1989} that $A(X)$ is a $k$-space if and only if $X$ is a locally compact space and the set of all non-isolated points of $X$ is separable. From Corollary~\ref{c0}, it is easy to see that there exists a space $Y$ which is not a $k$-space such that $A(Y)$ has a $\mathfrak{G}$-base.

However, the situation changes much for (non-Abelian) free topological groups. Let $X=C\bigoplus D$, where $C$ is a non-trivial convergent sequence with the limit point and $D$ is a closed discrete space of
cardinality $\aleph_{1}$. From \cite{LL}, $F_{4}(X)$ is not $csf$-countable, hence $F_{4}(X)$ does not have a $\mathfrak{G}$-base. In particular, we see that $F(X)$ does not have a $\mathfrak{G}$-base. However, we have the following Theorem~\ref{t2}.

By a similar proof of \cite[Proposition 2.7]{GKL}, we can obtain the following proposition.

\begin{proposition}\label{p0}
Suppose that, for each $n\in\mathbb{N}$, $X_{n}$ is a space with a $\mathfrak{G}$-base. Then the countable product $\prod_{n\in\mathbb{N}}X_{n}$ has a $\mathfrak{G}$-base.
\end{proposition}

Given a uniformizable space $X$ there is a finest uniformity on $X$ compatible with the topology of $X$ called the {\it fine uniformity} or {\it universal uniformity}.  A Tychonoff space $X$ is said to have a {\it uniform $\mathfrak{G}$-base} if there exists a uniform structure $\mathscr{U}$ on $X$, which induces the topology of $X$, such that $\mathscr{U}$ has a $\mathfrak{G}$-base. In particular, if $\mathscr{U}$ is the universal uniformity on $X$ with a uniform $\mathfrak{G}$-base, then we say that $X$ has an {\it universally uniform $\mathfrak{G}$-base}.

\begin{theorem}\label{t2}
Let $X$ have an universally uniform $\mathfrak{G}$-base. Then $F_{2}(X)$ has a $\mathfrak{G}$-base at each point.
\end{theorem}

\begin{proof}
Since $X$ has an universally uniform $\mathfrak{G}$-base, it is easy to see $X$ has a local $\mathfrak{G}$-base at each point. By Proposition~\ref{p0}, we see that $(X\bigoplus X^{-1}\bigoplus\{e\})^{2}$ has a local $\mathfrak{G}$-base at each point. Then $F_{2}(X)$ has a local $\mathfrak{G}$-base at each point $x\in X\cup X^{-1}$ since $X\cup X^{-1}$ is open and closed in $F_{2}(X)$. It is well known that $F_{2}(X)\setminus F_{1}(X)$ is homeomorphic to a subspace of $(X\bigoplus X^{-1}\bigoplus\{e\})^{2}$, and then $F_{2}(X)$ has a local $\mathfrak{G}$-base at each point $x\in F_{2}(X)\setminus F_{1}(X)$ since $F_{2}(X)\setminus F_{1}(X)$ is open in $F_{2}(X)$. It suffices to show that $F_{2}(X)$ has a $\mathfrak{G}$-base at $e$.

Suppose that $\mathscr{U}$ is the universally uniformity on $X$. Then one can take a basis $\mathscr{B}=\{U_{\alpha}: \alpha\in\mathbb{N}^{\mathbb{N}}\}$ for $\mathscr{U}$ such that for any $\alpha$ and $\beta$ in $\mathbb{N}^{\mathbb{N}}$ with $\alpha\leq\beta$, we have $U_{\beta}\subset U_{\alpha}$.

For each $\alpha\in\mathbb{N}^{\mathbb{N}}$, let $W_{\alpha}=\{x^{\varepsilon}y^{-\varepsilon}: (x, y)\in U_{\alpha}, \varepsilon=\pm 1\}.$ Then the family $\{W_{\alpha}: \alpha\in\mathbb{N}^{\mathbb{N}}\}$ is a base at $e$ in $F_{2}(X)$ by \cite{Y1998}. Obviously, $\{W_{\alpha}: \alpha\in\mathbb{N}^{\mathbb{N}}\}$ satisfies that for any $\alpha$ and $\beta$ in $\mathbb{N}^{\mathbb{N}}$ with $\alpha\leq\beta$, $W_{\beta}\subset W_{\alpha}$. Therefore, $\{W_{\alpha}: \alpha\in\mathbb{N}^{\mathbb{N}}\}$ is a local $\mathfrak{G}$-base at $e$.
\end{proof}

However, the following question is still unknown for us.

\begin{question}
Let $X$ be a space. If $F_{2}(X)$ has a $\mathfrak{G}$-base, does $F_{3}(X)$ have a $\mathfrak{G}$-base?
\end{question}

{\bf Acknowledgements}. The authors wish to thank professors Salvador Hern\'{a}ndez and Boaz Tsaban for telling us some information of the paper \cite{CVHT2014}. Moreover, the authors wish to thank professor Chuan Liu for reading parts of this paper and making comments. Finally, we hope to thank professor Shou Lin for finding a gap in our proof of Theorem 3.18 in the previous version and giving some key for us to supplement the proof.


\begin{thebibliography}{99}
\bibitem{A1986} A.V. Arhangel'ski\v\i,  {\it Hurewicz spaces, analytic sets and fan-tightness of spaces of functions}, Soviet Math. Dokl., {\bf 33(2)}(1986), 396--399.

\bibitem{AB1996} A.V. Arhangel'ski\v\i, A. Bella,  {\it Countable fan-tightness versus countable tightness}, Comment. Math. Univ. Carolinae, {\bf 37(3)}(1996), 567--578.

\bibitem{AT2008} A.V. Arhangel'ski\v{\i}, M.G. Tkachenko, {\it Topological Groups and Related Structures}, Atlantis Press and
World Sci., Paris, 2008.

\bibitem{AOP1989} A.V. Arhangel'ski\v\i, O.G. Okunev, V.G. Pestov,  {\it Free topological groups over metrizable spaces}, Topology Appl., {\bf 33}(1989), 63--76.

\bibitem{B} T. Banakh,  {\it $\mathfrak{P}_{0}$-spaces}, Topology Appl., {\bf 195}(2015), 151--173.

\bibitem{B1} T. Banakh,  {\it The strong Pytkeev property in topological spaces}, http://arXiv:1412.4268v1.

\bibitem{B2} T. Banakh,  {\it $\omega^{\omega}$-bases in topological and uniform spaces}, http://arxiv:1607.07978v1.

\bibitem{C2015} Z. Cai, S. Lin,  {\it Sequentially compact spaces with a point-countable $k$-network}, Topology Appl., {\bf 193}(2015), 162--166.

\bibitem{CVHT2014} C. Chis, M. Vincenta Ferrer, Salvador Hern\'{a}ndez, Boaz Tsaban,  {\it The character of topological groups, via bounded systems, Pontryagin-van Kampen duality and pcf theory}, J. Algebra, {\bf 420}(2014), 86--119.

\bibitem{D961} R.M. Dudley, {\it Continuity of homomorphisms}, Duke Math. J., {\bf 28}(1961), 587--594.

\bibitem{E1989} R. Engelking, {\it General Topology} (revised and completed edition),
  Heldermann Verlag, Berlin, 1989.

\bibitem{F2006} J.C. Ferrando, J. K\c{a}kol, M. L\'{o}pez Pellicer, S.A. Saxon, {\it Tightness and distinguished
Fr\'{e}chet spaces}, J. Math. Anal. Appl., {\bf 324}(2006), 862--881.

\bibitem{F1982} P. Fletcher, W.F. Lindgren, Quasi-uniform spaces, Marcel Dekker, New York, 1982.

\bibitem{F1960} Z. Frol\'{\i}k, {\it Generalizations of the $G$-property of complete metric spaces}, Czech.
Math. J., {\bf 10}(1960), 359--379.

\bibitem{GKL} S.S. Gabriyelyan, J. K\c{a}kol, A. Leiderman,  {\it On topological groups with a small base and
metrizability}, Fund. Math., {\bf 229}(2015), 129--158.

\bibitem{GKL2} S.S. Gabriyelyan, J. K\c{a}kol, A. Leiderman,  {\it The strong Pytkeev property for topological groups
and topological vector spaces}, Monatsh Math., {\bf 175}(2014), 519--542.

\bibitem{GKL3} S.S. Gabriyelyan, J. K\c{a}kol, {\it On topological spaces and topological groups with certain local countable networks}, Topology Appl., {\bf 190}(2015), 59--73.

\bibitem{GKL1} S.S. Gabriyelyan, J. K\c{a}kol, A. Kubzdela and M. Lopez Pellicer,  {\it On topological properties of Fr\'{e}chet
locally convex spaces with the weak topology}, Topology Appl., {\bf 192}(2015), 123--137.

\bibitem{GKL4} S.S. Gabriyelyan, J. K\c{a}kol, {\it On $\mathfrak{B}$-space and related concepts}, Topology Appl., {\bf 191}(2015), 178--198.

\bibitem{G1962}  M.I. Graev, {\it Free topological groups}, In: Topology and Topological Algebra,
Translations Series 1, vol. 8 (1962), pp. 305--364.  American Mathematical Society.
Russian original in: Izvestiya Akad. Nauk SSSR Ser. Mat., {\bf 12}(1948), 279--323.

\bibitem{Gr} G. Gruenhage, {\it Generalized metric spaces}, In: K. Kunen, J. E.
  Vaughan(Eds.), Handbook of Set-Theoretic Topology, Elsevier Science
  Publishers B.V., Amsterdam, 1984, 423--501.
  
\bibitem{GMT1} G. Gruenhage, E.A Michael, Y. Tanaka, {\it Spaces determined by point-countable covvers},
Pacific J. Math., {\bf 113}(1984), 303--332.
  
\bibitem{G1982} G. Gruenhage, Y. Tanaka, {\it Products of $k$-spaces and
spaces of countable tightness}, Trans. Amer. Math. Soc., {\bf 273}(1982), 299--308.

  \bibitem{G1971} J.A. Guthrie,  {\it A characterization of $\aleph_{0}$-spaces},
  General Topology Appl., {\bf 1}(1971), 105--110.

 \bibitem{K1985} Y. Kanatani, N. Sasaki, J. Nagata, {\it New characterizations of some generalized
metric spaces}, Math Japonica, {\bf 30}(1985), 805--820.

  \bibitem{LPT2015} A.G. Leiderman, V.G. Pestov , A.H. Tomita,  {\it On topological groups admitting a base at indentity indexed with $\omega^\omega$},
http://arxiv: 1511.07062v1.

\bibitem{LLL} Z. Li, F. Lin, C. Liu,  {\it Networks on free topological groups}, Topology Appl., {\bf 180}(2015), 186--198.

\bibitem{LL} F. Lin, C. Liu, J. Cao,  {\it Weak Countability Axioms in Free Topological Groups}, submitted.

\bibitem{LT1994} S. Lin, Y. Tanaka, {\it Point-countable $k$-networks, closed maps,
  and related results}, Topology Appl., {\bf 59}(1994), 79--86.

\bibitem{MA1945} A.A. Markov, {\it On free topological groups}, Izv. Akad. Nauk
  SSSR Ser. Mat., {\bf 9}(1945), 3--64 (in Russian); Amer. Math. Soc. Transl.,
  {\bf 8}(1962), 195--272.

\bibitem{T2003} P. Nickolas, M. Tkachenko, {\it Local compactness in free topological groups}, Bull. Austral. Math. Soc.,  {\bf 68}(2)(2003), 243--265.

  \bibitem{O1971} P. O$^{\prime}$Meara, {\it On paracompactness in function spaces
  with the compact-open topology}, Proc. Amer. Math. Soc.,{\bf 29}(1971),
  183--189.

\bibitem{P1983} E.G. Pytkeev, {\it Maximally decomposable spaces}, Trudy Mat. Inst. Steklov., {\bf 154}(1983),
209--213.

\bibitem{S2000} O.V. Sipacheva, {\it Free topological groups of spaces and their subspaces}, Topology Appl., {\bf 101}(2000),
181--212.

\bibitem{Sn} V. \v Sne\v\i der, {\it Continuous images of Souslin and Borel sets; metrization theorems}, Dokl. Acad. Nauk USSR, {\bf 50}(1945), 77--79.

\bibitem{T1984} M. G. Tkachenko, {\it On a spectral decomposition of free topological groups}, Usp. Mat. Nauk, {\bf 39(2)}(1984), 191--192.

\bibitem{TZ2009} Boaz Tsaban, L. Zdomskyy, {\it On the Pytkeev property in spaces of continuous functions (II)}, Houston J. Math., {\bf 35}(2009),
563--571.

\bibitem{Y1993} K. Yamada, {\it Characterizations of a metrizable space such that every $A_n(X)$ is a $k$-space}, Topology Appl., {\bf 49}(1993), 74--94.

\bibitem{Y1997} K. Yamada, {\it Tightness of free Abelian topological groups and of finite product of sequntial fans}, Topology Proc., {\bf 22}(1997), 363--381.

\bibitem{Y1998} K. Yamada, {\it Metrizable subspaces of free topological groups on metrizable spaces}, Topology Proc., {\bf 23}(1998), 379--409.

\bibitem{Y2005} K. Yamada, {\it The natural mappings $i_{n}$ and $k$-subspaces of free topological groups on metrizable spaces}, Topology Appl., {\bf 146-147}(2005), 239--251.
\end{thebibliography}
\end{document}